\documentclass[12pt]{amsart}
\pdfoutput=1
\usepackage[english]{babel}
\usepackage[cp850]{inputenc}
\usepackage{amscd}
\usepackage{amssymb}
\usepackage{amsmath}
\usepackage{graphics}
\usepackage{graphicx}
\usepackage[normalem]{ulem}
\usepackage{cancel}
\newtheorem{theo}{Theorem}

\newtheorem{lem}{Lemma}
\newtheorem{prop}{Proposition}

\newtheorem{rem}{Remark}

\author{}

\begin{document}

\begin{center}
{\bf\Large  The 2-center problem and ball operators in strictly convex  normed planes}\\[2ex]
by\\[2ex]
{\sc Pedro Mart\'{\i}n, Horst Martini, and Margarita Spirova }
\end{center}
{

\vspace*{3ex}

{\bf Abstract.} We investigate the $2$-center problem for arbitrary strictly convex, centrally symmetric curves instead of usual circles. In other words, we extend  the $2$-center problem (from the Euclidean plane) to strictly convex normed planes, since any  strictly convex, centrally symmetric curve can be interpreted as (unit) circle of such a normed plane. Thus we generalize  the respective algorithmical approach given by J. Hershberger for the Euclidean plane. We show  that the corresponding decision problem can be solved in $O(n^2\log\, n)$  time. In addition, we prove various theorems on the notions of ball hull and ball intersection of finite sets in strictly convex normed planes, which are fundamental for the $2$-center problem, but also  interesting for themselves.


\thanks{}

\title{}

\maketitle

\vspace{-5ex}

\section{Introduction}

The \emph{$2$-center problem} can be formulated as follows: given a (finite) set $K$ of points in the plane,  find two congruent closed discs of smallest possible radius and with centers from $K$ whose union covers $K$. Variants of this problem are  frequently  investigated, e.g., by allowing the centers of the covering discs to be also not from $K$, or by considering two different radii (still suitably restricted). For Euclidean discs many results on different aspects of this problem were derived, particularly from the viewpoint of computational geometry (see \cite{Ag-Sh-SWe},   \cite{H-S-G}, \cite{H-S}, \cite{H2}, and \cite{Sh}). It is our aim to extend this problem  to arbitrary strictly convex normed planes. Geometrically this means that we investigate the $2$-center problem for arbitrary strictly convex, centrally symmetric curves instead of Euclidean circles. We note that this is the first approach to the 2-center problem for a large class of norms. The only contributions for non-Euclidean norms until now refer to a special norm (see  \cite{KKS} and \cite{KC}).
In addition, we study geometric properties of planar ball hulls and ball intersections, which are shown to be helpful for investigating (also algorithmically) the planar $2$-center problem, where all these notions and problems are again defined in the sense of strictly convex normed planes. The \emph{ball hull}, respectively the \emph{ball intersection}, of a given set $K$ in a normed plane is the intersection of all norm discs of suitably fixed radius that contain $K$, respectively the intersection of all  norm discs of suitably  fixed radius whose centers are from $K$. The notions of ball hull and ball intersection are important in Banach-space theory; they are basic for investigations on circumballs, Chebyshev sets, complete sets, bodies of constant width, ball polytopes,  and Jung constants (see, e.g., \cite{Al-Ma-Sp}, \cite{Grue1}, \cite{Ban}, \cite{Bor}, \cite{Ch-L}, \cite{MRS}, \cite{MPS}, \cite{KMP}, and \cite{Ma-Ma-Sp}).

 From the more theoretical point of view we obtain   various new results  on the  ball hull and the ball intersection of a finite set $K$ in arbitrary strictly convex normed planes. These  results   shed more light on the geometric structure of the boundary of such point sets associated to $K$. Several lemmas and theorems here can be considered as natural sharpenings  of results from \cite{Ma-Ma-Sp} for the planar case (see also \cite{Ma-Sp}).
   Further on, these results can be directly used for  the construction of planar ball hulls and ball intersections via circular arcs (in the sense of the respective norm). This can be constructively used for the  algorithmical approach to the $2$-center problem, presented in the final section of our paper. The algorithm developed there is a natural extension of that of J. Hershberger  for the Euclidean subcase  (cf.  \cite{H2} and  also \cite{H-S}). Finally, we prove  that the $2$-center decision problem, for two suitably different radii and arbitrary strictly convex norms, can be solved in $O(n^2\log\, n)$ time and $O(n^2)$ space.

\medskip

\bigskip
\section{Notation and preliminaries}
\medskip

Let $\mathbb{M}^d=(\mathbb{R}^d, \|\cdot\|)$ be a $d$-dimensional  \emph{normed} (or \emph{Minkowski}) \emph{space}. As is well-known, the \emph{unit ball} $B$ of  $\mathbb{M}^d$ is   a compact, convex set with non-empty interior (i.e., a \emph{convex body}) centered at the \emph{origin} $o$. The boundary of $B$ is the \emph{unit sphere of} $\mathbb{M}^d$, and any  homothetical copy $x+\lambda B$ of $B$  is called the  \emph{ball with center $x\in\mathbb{R}^d$ and radius} $\lambda > 0$ and  denoted by $B(x,\lambda)$; its boundary is the  \emph{sphere}  $S(x,\lambda).$ We use the usual abbreviations $\mathrm{int}$  for the \emph{interior} and $\mathrm{conv}$ for the \emph{convex hull}.  The \emph{line segment} connecting  two distinct points $p$ and $q$ is denoted by $\overline{pq}$, and its affine hull is the  \emph{line}  $\langle p, q\rangle$.

Given a set of points $K$ in $\mathbb{M}^2$ and $\lambda>0$, the \emph{$\lambda$-ball hull} $\operatorname{bh}(K, \lambda)$ of $K$ is defined as the intersection of all  balls of radius $\lambda$ that contain $K$:
$$
\operatorname{bh}(K, \lambda)=\bigcap_{K\subset B(x,\lambda)}B(x, \lambda).
$$
Our Theorem \ref{theo1} below describes the boundary structure of $\mathrm{bh}(K,\lambda)$.

The \emph{$\lambda$-ball intersection} $\operatorname{bi}(K, \lambda)$ of $K$ is the intersection of all  balls of radius $\lambda$ whose centers are from $K$:
$$
\operatorname{bi}(K, \lambda)=\bigcap_{x\in K}B(x,\lambda).
$$
Note that for $d=2$ and $K$ a finite set, the boundary structure of $\mathrm{bi}(K,\lambda)$ consists of circular arcs of radius $\lambda$ with centers belonging to $K$.

  Naturally, these notions make  sense only  if $\operatorname{bh}(K, \lambda)\neq\emptyset$ and $\operatorname{bi}(K, \lambda)\neq\emptyset$. It is clear that $\operatorname{bh}(K, \lambda)\neq\emptyset$ if and only if $\lambda\geq\lambda_K$, where $\lambda_K$ is the smallest number such that $K$ is contained in a translate of $\lambda_K B$. Such a translate is called a \emph{minimal enclosing ball of $K$} (or a \emph{circumball of $K$}), and $\lambda_K$ is said to be the \emph{circumradius} (or  \emph{Chebyshev radius}) \emph{of} $K$. Clearly, we have
 \begin{equation}K_1\subseteq K_2\;\Longrightarrow \lambda_{K_1}\leq \lambda_{K_2}.\label{11} \end{equation}
 In the Euclidean subcase the minimal enclosing ball of a bounded set is always unique, but this is no longer true for an arbitrary norm. It is  easy to check that for an arbitrary norm
\begin{equation}\begin{split}\{x\in \mathbb{M}^d: x\;\text{is the center of a minimal enclosing disc of}\; K\}\\=\operatorname{bi}(K, \lambda_K),\end{split}\label{19}\end{equation} yielding that
  $\operatorname{bi}(K, \lambda)\neq\emptyset$ if and only if $\lambda\geq\lambda_K$. The set of centers of minimal enclosing balls of $K$ is called the \emph{Chebyshev set of $K$}.  Note that, in contrast to the Euclidean situation, in general normed spaces the Chebyshev set of a bounded set does not  necessarily belong  to the convex hull of this set. 
  In Figure \ref{fig1}, considering $K$ as the set of vertices of the shown triangle, the Chebyshev set of $K$ does not belong to its convex hull.

\begin{figure}[h]
\begin{center}
\parbox[]{5cm}{
\scalebox{1}{\includegraphics[viewport=154 70 339 260,clip]{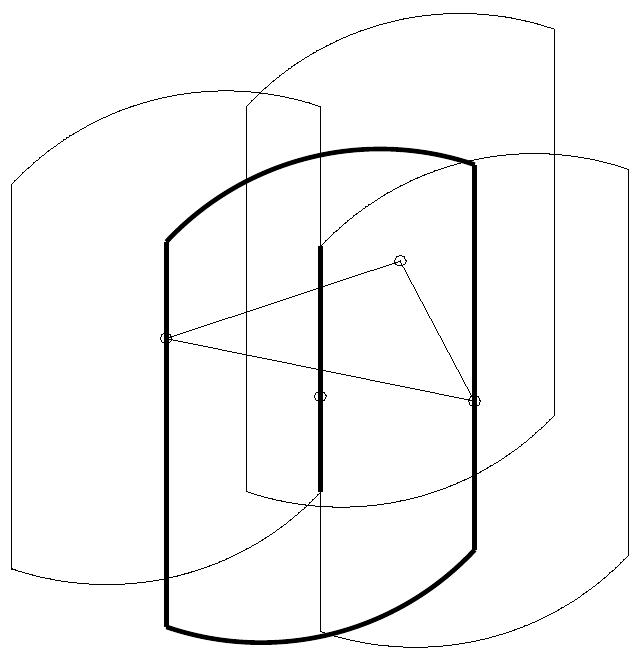}}\\}
\caption{$\operatorname{bi}(K, \lambda_K)$ does not belong to the convex hull of $K$.}\label{fig1}
\end{center}
\end{figure}

For  a  compact set $K$ in $\mathbb{M}^d$  denote  by $\operatorname{diam} (K)
:=\max\{\|x-y\|: x, y\in K\}$  the \emph{diameter of } $K$.

 In what follows, when we speak about the $\lambda$-ball intersection or $\lambda$-ball hull of a set $K$, we always mean that $\lambda\geq\lambda_K$.
  It is easy to check that
  \begin{equation}\label{12}\lambda_K\leq \mathrm{diam} (K) \leq 2 \lambda_K;\end{equation} see also \cite{Bar-Pap}.

A non-zero vector $y\in \mathbb{M}^2$ is \emph{Birkhoff orthogonal} to a non-zero vector $x\in \mathbb{M}^2$ if for any real $\lambda$ the inequality $\|y\|\leq \|y+\lambda x\|$ holds.

\smallskip

\bigskip

\bigskip
\section{On ball hulls of  finite  sets in the planar case}\label{section3}
\medskip

In this section we refer to normed planes, using notions like circles and discs instead of spheres and balls, respectively. On the other hand, we  stay with the notions of ball hull and ball intersection also for the planar case. Note also that further on, \emph{translates of the unit disc} (or  \emph{translates of the unit circle}) are, for simplicity, also  called \emph{unit discs} (\emph{unit circles}).

Let $p$ and $q$  be two points of  the circle $S(x,\lambda)$. In the following, the \textit{minimal circular arc of $B(x,\lambda)$ meeting $p$ and $q$ } is the piece of $S(x,\lambda)$  with endpoints $p$ and $q$, which lies in the half-plane bounded by the line $\langle p, q\rangle$ and   not containing the center $x$. If $p$ and $q$ are opposite in $S(x,\lambda)$, then the two half-circles with endpoints $p$ and $q$ are minimal circular arcs of $S(x,\lambda)$ meeting $p$ and $q$. A minimal circular arc meeting $p$ and $q$ we denote by  $\widehat{pq}$.

\smallskip

 Our objective in this section is to describe the structure of the ball hull of a finite set $K$ (see Theorem \ref{theo1} and Theorem \ref{bi-bh}).

Gr\"{u}nbaum \cite{Grue1} and Banasiak \cite{Ban} proved the following lemma  (see also \cite[$\S$ 3.3]{MSW} and \cite{Al-Ma-Sp}).

\begin{lem}\label{twocircles} Let $\mathbb{M}^2$ be a normed plane. Let $C\subset \mathbb{M}^2$ be a compact, convex disc whose  boundary is the closed curve $\gamma$; $v$ be a vector in $ \mathbb{M}^2$; $C'=C+v$ be a translate of $C$ with boundary $\gamma'$. Then $\gamma\cap \gamma'$ is the  union of two segments, each of which may degenerate to a point or to the empty set.

Suppose that this intersection consists of two connected non-empty components $A_1$, $A_2$. Then the two lines parallel to the line
of translation and supporting $C\cap C'$ intersect $C\cap C'$  exactly in $A_1$ and $A_2.$

Choose a point $p_i$ from each component $A_i$ and let $c_i=p_i-v$ and $c_i'=p_i+v$ for $i=1,2.$ Let $\gamma_1$ be the part of $\gamma$ on the same side of the line $\langle p_1, p_2\rangle$ as $c_1$ and $c_2$; let $\gamma_2$ be the part of $\gamma$ on the  side of  $\langle p_1, p_2\rangle$ opposite to  $c_1$ and $c_2$; and similarly for $\gamma'$, $\gamma_1'$, and $\gamma_2'$.

Then $\gamma_2\subseteq \operatorname{conv}(\gamma_1')$ and $\gamma_2'\subseteq \operatorname{conv}(\gamma_1).$
\end{lem}


Our next  lemma holds only for strictly convex normed planes.

\begin{lem}\label{3.0}
Let $\mathbb{M}^2$ be a strictly convex normed plane  with unit disc $B$ and $p, q\in B$. Then:
 \begin{enumerate}
 \item If  $p, q\in S(o,1)$ and there exists another circle $S(x,1)$ through $p$ and $q$, then the origin $o$ and $x$ are in different half-planes with respect to the line $\langle p, q\rangle$.

 \item Any minimal circular arc of radius $1$ meeting $p$ and $q$ also belongs to $B$.

 \item  If a circular arc of radius 1 meeting $p$ and $q$ is  contained in $B$ and contains interior points of $B$, then this arc is a minimal circular arc.
 \end{enumerate}
\end{lem}

We note that this lemma is presented as Lemma 5.2 in \cite{Ma-Ma-Sp} \emph{for all normed planes} adding the condition $\|p-q\|\leq 1$, although then the statement (1) gets a second case (not possible for strictly convex norms): the segment $\overline{pq}$ belongs to $S(x, 1)\cap S(o, 1)$. Discs of radius $1$ are considered only for simplicity, and therefore both Lemma 5.2 in \cite{Ma-Ma-Sp} and Lemma \ref{3.0} here are obviously true for discs with an arbitrary radius $\lambda$.

\bigskip

\emph{Proof of Lemma \ref{3.0}:}
The points $o,x,p+q$ belong to $S(p,1)\cap S(q,1)$. By Lemma \ref{twocircles} this intersection is the union of two segments, each of which may degenerate to a point or to the empty set. Therefore, because the normed plane is strictly convex, we obtain $x=p+q$ and $(1)$ follows.

 $(2)$ Let us consider a minimal circular arc of radius $1$ meeting $p$ and $q$, and let $S(x,1)$ be the circle that contains this arc. The curves $C=S(o,1)$ and $C'=S(x,1)$ satisfy the hypothesis of Lemma \ref{twocircles}. Let $p_i$, $\gamma_i$ and $\gamma_i'$ (with $i=1,2$) be as in Lemma \ref{twocircles}. Since the plane is strictly convex, there exist two points $p_1$ and $p_2$ such that the component $\gamma_2'$ is maximal. Then the points $p$ and $q$ and the minimal circular arc meeting $p$ and $q$ are in $\gamma_2'$. By Lemma \ref{twocircles} we have $\gamma_2'\subseteq \operatorname{conv}(\gamma_1)\subseteq B(0,\lambda).$

$(3)$ Let us assume that $p$ and $q$ are in $B$ and there exists an arc $A$ of radius 1 meeting $p$ and $q$   which has points in  $\mathrm{int} B$. Let $B(x,1)$ be the ball such that $A$ is in $S(x,1)$.  We have $x\neq o$, because there are points of $A$ in  $\mathrm{int} B$. The set  $S(o,\lambda)\cap S(x,\lambda)$ is the union of two points, each of which may degenerate to the empty set. If the intersection consists of exactly two points $p_1$ and $p_2$ like in Lemma \ref{twocircles}, then there is a component $\gamma_2$ of $S(x,1)$ defined by $p_1$ and $p_2$ and   satisfying the following: it is inside $B$, and the other component of $S(x,1)$ has only the two points $p_1$ and $p_2$ in $B$. The points $p$ and $q$ are in $B$,  and, by the above result, the minimal arc of $S(x,1)$ meeting $p$ and $q$ is also in $B$ . Then, this minimal arc is in $\gamma_2$, which implies that the other arc defined in $S(x,1)$ by $p$ and $q$ has points outside of  $B$.\hfill$\Box$

From the above, we  obtain the following.

\begin{lem}\label{3.2}
Let $\mathbb{M}^2$ be a normed plane. Let $p$ and $q$ be two points belonging to a disc of radius $\lambda$. If $\mathbb{M}^2$ is strictly convex or $\|p-q\|\leq \lambda$, then
\begin{enumerate}
\item there exist only two minimal arcs of radius $\lambda$ meeting $p$ and $q$ (which may degenerate to the segment $\bar{pq}$),
\item every disc of radius $\lambda$ containing $p$ and $q$ also contains the minimal circular arcs of radius $\lambda$ meeting $p$ and $q$,
\item for every $\alpha\geq \lambda$, each disc of radius $\lambda$ containing $p$ and $q$ also contains the minimal circular arcs of radius $\alpha$ meeting $p$ and $q$,
\item if a circular arc of radius $\lambda$ meeting $p$ and $q$ is contained in a disc of radius $\lambda$ and contains interior points of
the disc, then this arc is a minimal circular arc.
\end{enumerate}

\end{lem}

\bigskip

\begin{proof}
Let us assume that $\|p-q\|\leq \lambda$. Parts $(1)$, $(2)$ and $(4)$ follow directly from
\cite[Lemma 5.2]{Ma-Ma-Sp}, where the assumption $\lambda =1$ is considered
only for simplicity.

 By $(1)$ and $(2)$, the ball hull of the set $\{p,q\}$ with radius $\lambda$ is the set bounded by the two minimal circular arcs of radius $\lambda$ meeting $p$ and $q$. If $\alpha\geq \lambda$, then $\mathrm{bh}(\{p,q\}, \alpha) \subseteq \mathrm{bh}(\{p,q\}, \lambda)$ because the ball hull operator is non-increasing with respect to radii (\cite[Proposition 3.1]{Ma-Ma-Sp}), and we obtain (3).

If $\mathbb{M}^2$ is strictly convex, then it follows that the result is also true when the distance $\|p-q\|$ is larger than $\lambda$, using the same arguments and Lemma \ref{3.0}.

\end{proof}

The above result  is not true for normed planes which are not strictly convex and satisfy $\|p-q\|>\lambda$. With the maximum norm, the unit disc centered at $o=(0,0)$, the unit disc centered at $x=(0.1,0)$ and the points $p=(0.5,1)$ and $q=(0.5,-1)$ yield a counterexample.
\bigskip

We get immediately

\begin{lem}\label{4}
Let $K=\{p_1,p_2,\dots,p_n\}$ be a finite set in a normed plane $\mathbb{M}^2$ and $\lambda>0$. If either the plane $\mathbb{M}^2$ is strictly convex or $\lambda \geq \mathrm{diam}(K)$, then any disc $B(x,\lambda)$ which contains $K$ also contains every minimal circular arc of radius larger than or equal to $\lambda$ and meeting $p_i$ and $p_j$ ($i,j=1,2,\dots, n$).
\end{lem}

\bigskip
In particular, we get the following lemma for triples of points.

\begin{lem}\label{lemma5}
Let $p_1$, $p_2$ and $p_3$ be three points in a normed plane $\mathbb{M}^2$  and $\lambda>0$.  If either the plane $\mathbb{M}^2$ is strictly convex or $\lambda \geq \mathrm{diam}(\{p_1,p_2,p_3\})$, and
\begin{itemize}
\item there exists a point $x_{12}$ and a minimal circular arc $\widehat{p_1p_2}$ contained in $S(x_{12},\lambda)$ such that $p_3$ is an interior point of $B(x_{12},\lambda)$,
\item there exists a point $x_{23}$ and a minimal circular arc $\widehat{p_2p_3}$ contained in $S(x_{23},\lambda)$ such that $p_1$ is an interior point of $B(x_{23},\lambda)$,
\item $p_1$, $p_2$ and $p_3$ are not in a line,
\end{itemize}
then $p_1\notin \operatorname{conv}(\widehat{p_2p_3}, \overline{p_2p_3})$ and $p_3\notin\operatorname{conv}(\widehat{p_1p_2}, \overline{p_1p_2})$.
\end{lem}

\smallskip

\begin{proof} Let us assume  that $p_3\in \operatorname{conv}(\widehat{p_1p_2}, \overline{p_1p_2})$. Since $p_1$ is an interior point of $B(x_{23},\lambda)$, then $S(x_{23},\lambda)$ meets $\overline{p_1p_2}$ or $\widehat{p_1p_2}$ in a point different from $p_2$. In any case, the arc $\widehat{p_1p_2}$ is not contained in $B(x_{23},\lambda)$, in contradiction to Lemma \ref{4}. For similar reasons, $p_1\notin \operatorname{conv}(\widehat{p_2p_3}, \overline{p_2p_3})$.

\end{proof}

\bigskip

\begin{lem}\label{5}
Let $p$ and $x$ be two points in a normed plane $\mathbb{M}^2$. Let $\lambda = \|p-x\|$ and $\alpha$ be such that $\lambda<\alpha\leq 2\lambda$. Let $a$ and $b$ be the two points in $S(x,\lambda)\cap S(p,\alpha)$ such that the arc meeting $b$ and $a$ (clockwise) in $S(x,\lambda)$  not containing $p$ is maximal (namely, there is not another arc satisfying the same  conditions larger than this one). Then, the arcs $\widehat{pb}$ and $\widehat{ap}$ (clockwise) in $S(x,\lambda)$  are minimal circular arcs of radius $\lambda$, and any arc contained in the arcs $\widehat{pb}$ or $\widehat{ap}$ is also minimal.

\end{lem}

\smallskip

\begin{proof}
If $\lambda<\alpha<2\lambda$, then $S(x,\lambda)$ is not contained in $B(p,\alpha)$, and the point symmetric to $p$  with respect to $x$ does not belong to $B(p,\alpha)$. The line $\langle p , b \rangle$ separates the arc $\widehat{pb}$ (in  clockwise direction) in $S(x,\lambda)$ and the center $x$. Therefore, the arc $\widehat{pb}$ is a minimal circular arc of radius $\lambda$  meeting $p$ and $b$. Similarly for the arc $\widehat{ap}$.

If $\alpha=2\lambda$, the  point symmetric  to $p$ with  respect to $x$ is in $S(x,\lambda)\cap S(p,\alpha)$ between $a$ and $b$, and either the line $\langle p , b\rangle$ separates the arc $\widehat{pb}$ and the point $x$, or the point $x$ belongs to the line. In any case, the arc $\widehat{pb}$ is a minimal circular arc of radius $\lambda$. Similarly for the arc $\widehat{ap}$.
\end{proof}

\bigskip


\begin{prop} (Proposition 5.5 in \cite{Ma-Ma-Sp})\label{proposition5.5}
Let $K=\{p_1,p_2,\dots,p_n\}$ be a finite set in  a normed plane $\mathbb{M}^2$ having diameter $1$. Then
$$\operatorname{bh}(K)=\bigcap_{i=1}^k B(x_i,1),$$
where $B(x_i,1)$, $i=1,2,...,k$, are discs which contain $K$, and whose spheres contain some minimal arcs meeting points of $K$.
\end{prop}

In the proof of Proposition \ref{proposition5.5}, balls of radius 1 are considered only for simplicity. Furthermore, the proof really requires that only balls of radius $\lambda$ greater than or equal to the diameter of $K$ are considered.

\begin{prop} \label{proposition3}
Let $K=\{p_1,p_2,\dots,p_n\}$ be a finite set in  a normed plane $\mathbb{M}^2$, and let $B(x_i,\lambda)$, $i=1,\dots,k$, be discs of radius $\lambda$ which contain $K$ and whose circular boundaries contain some minimal arcs of radius $\lambda$ meeting points of $K$. If $\mathbb{M}^2$ is strictly convex or $\lambda\geq \mathrm{diam}(K)$, then
$$\operatorname{bh}(K,\lambda)=\bigcap_{i=1}^k B(x_i,\lambda).$$

\end{prop}

\begin{proof} Let $\lambda\geq \mathrm{diam}(K)$. By considering discs of radius $\lambda$ instead of radius $1$, and Lemmas \ref{3.2}, \ref{4} and \ref{lemma5}, the statement can be proven by rewriting exactly the same proof of Proposition \ref{proposition5.5} presented in \cite[Proposition 5.5]{Ma-Ma-Sp}.




Let now $\mathrm{diam}(K)>\lambda\geq \lambda_K$ and $\mathbb{M}^2$ be strictly convex. Since  $2\lambda_K\geq\mathrm{diam}(K)$ (see inequality (\ref{12})), then
$$
2\lambda \geq 2\lambda_K \geq \mathrm{diam}(K) > \lambda \geq \lambda_K.
$$
We also use the proof of Proposition \ref{proposition5.5} presented in \cite{Ma-Ma-Sp}, including slight modifications. Analogously, we fix the clockwise orientation of a closed curve in $\mathbb{M}^2$ as the \emph{ negative orientation} of that curve.  Since $\lambda_K\leq \lambda$, there exists a disc $B(x_1,\lambda)$ such that $K\subset B(x_1,\lambda)$. After translating  and renaming the points if  necessary, we may assume that
 \begin{itemize}
 \item $S(x_1,\lambda)$ contains two points $p_1, p_2\in K$,
 \item the circular arc starting at $p_1$  with negative orientation and  ending in $p_2$ is a minimal circular arc,
 \item there is no other minimal circular arc in $S(x_1,\lambda)$ meeting points of $K$ and being larger than the minimal circular arc meeting $p_1$ and $p_2$.
 \end{itemize}

We denote by $\alpha$  the diameter of $K$. Then the set $K$  is contained in $B(p_2,\alpha)\cap B(x_1,\lambda)$. Starting in $z=x_1$, we move $z$ along $S(p_2,\lambda)$ in the negative direction. Let $x_2$ denote  the first position of $z$ such that one of the following conditions is verified:
 \begin{enumerate}
 \item There is a new point $p_3\in S(x_2,\lambda)$ such that
     \begin{itemize}
    \item the circular arc in $S(x_2,\lambda)$ starting in $p_2$  with negative orientation and  ending in $p_3$ is a minimal circular arc,
   \item there is no other minimal circular arc in $S(x_2,\lambda)$ meeting points of $K$ and being larger than the minimal circular arc meeting $p_2$ and $p_3$,
    \end{itemize}
 \item $p_1 \in S(x_2,\lambda)$ with $x_2$ from  the other half-plane defined by the line $\langle p_1, p_2\rangle$.
 \end{enumerate}

Lemma \ref{5} guarantees that these are the only two possible situations.

 In both cases, we consider the set $A=B(x_1,\lambda)\cap B(x_2,\lambda)$. Since $z$ moves continuously in $S(p_2,\lambda)$, $A$ contains $K$.
The rest of the proof can be carried over from the proof of Proposition \ref{proposition5.5} in  \cite[Proposition 5.5]{Ma-Ma-Sp} word by word, replacing the radius 1 of circles and discs everywhere by $\lambda$, and using the previous lemmas in this section. This yields the set
$$A=\bigcap_{i=1}^kB(x_i,\lambda)$$
as ball hull of $K$, because it is the intersection of discs with radius $\lambda$ which contain $K$, and its boundary is generated by minimal arcs meeting points of $K$.

\end{proof}

With the help of the above results  we can formulate the following theorem.

\bigskip

\begin{theo}\label{theo1}
Let $K=\{p_1,p_2,\dots,p_n\}$ be a finite set in  a normed plane  $\mathbb{M}^2$,  and let $\lambda\geq \lambda_K$. We denote by  $\widehat{p_ip_j}$  a minimal circular  arc of radius $\lambda$ meeting $p_i$ and $p_j$. Let $\mathcal{H}$ be the set of all discs of radius $\lambda$ such that their boundary contains a circular  arc meeting points from $K$. If the plane $\mathbb{M}^2$ is strictly convex or $\lambda \geq \mathrm{diam}(K)$, then
$$
\operatorname{bh}(K,\lambda)=\bigcap_{K\subset B(x,\lambda)\in \mathcal{H}}B(x,\lambda)=\operatorname{conv}(\bigcup_{i,j=1}^n \widehat{p_ip_j}).
$$
\end{theo}

\bigskip

Theorem \ref{theo1} shows that the boundary of the ball hull consists of arcs meeting points from $K$. The endpoints of such arcs are in fact extreme points of the ball hull. We call these points \emph{vertices of ball hulls}.

In \cite{J-S}  the following statement is proved.

\begin{lem}\label{curvebisector} Let $\mathfrak{B}(p, q)=\{x\in \mathbb{R}^2: \|x-p\|=\|x-q\|\}$ be the bisector of two points $p$ and $q$. Then there exists a curve $\mathfrak{B}^\ast(p, q)\subseteq \mathfrak{B}(p, q)$ (through the midpoint of $p$ and $q$ and symmetric with respect to this midpoint) which is homeomorphic to $\mathbb{R}$  such that for every $x\in \mathfrak{B}^\ast(p, q)$ the curve $\mathfrak{B}^\ast(p, q)$ belongs to the  double cone with apex $x$ and passing through $p$ and $q$. This curve separates the plane into two parts $\mathfrak{B}^\ast(p, q)^+$ and $\mathfrak{B}^\ast(p, q)^{-}$ such that whenever $y\in \mathfrak{B}^\ast(p, q)^+$, then  $\|p-y\|\leq\|q-y\|$, and whenever $y\in \mathfrak{B}^\ast(p, q)^{-}$, then $\|p-y\|\geq\|q-y\|$. Moreover, if $x\in \mathfrak{B}(p, q)$, then the curve $\mathfrak{B}^\ast(p, q)$ can be constructed to pass through $x$.
\end{lem}

Recall that the boundary of a planar ball intersection consists of circular arcs. Similarly to vertices of ball hulls, we call their endpoints \emph{vertices of ball intersections}.

\begin{theo}\label{bi-bh}
 Let $K=\{p_1,p_2,\dots,p_n\}$ be a finite set in  a normed plane  $\mathbb{M}^2$ and $\lambda\geq \lambda_K$. If $\mathbb{M}^2$ is strictly convex or $\lambda\geq \mathrm{diam}(K)$, then every arc of the boundary of $\mathrm{bi}(K,\lambda)$ has a vertex of $\mathrm{bh}(K,\lambda)$ as center. Moreover, every vertex of $\mathrm{bi}(K,\lambda)$ is the center of an arc belonging to the boundary of $\mathrm{bh}(K,\lambda).$
\end{theo}

\begin{proof}
Our objective is to prove the following: if $p\in K$ is not a vertex of  $\mathrm{bh}(K,\lambda)$, then $B(p,\lambda)$ has no influence on the construction of $\mathrm{bi}(K,\lambda)$ because $B(p,\lambda)$ contains the intersection of other discs whose centers are vertices of $\mathrm{bh}(K,\lambda)$.

By Proposition \ref{proposition3},  a point $p\in K$ which is not a vertex of $\mathrm{bh}(K,\lambda)$  can either  belong only to  $\mathrm{conv}(\widehat{p_ip_j},\overline{p_ip_j})$, or to $\mathrm{conv}(p_i, p_j, p_k)$, for some $p_i, p_j, p_k\in K$ which are vertices of $\mathrm{bh}(K,\lambda)$.


Let us assume that $p\in \mathrm{conv}(p_i, p_j, p_k)$. There exist some positive numbers $t_i,t_j,t_k$ such that $p=t_ip_i+t_jp_j+t_kp_k$ and $t_i+t_j+t_k=1$. Let $x$ be a point belonging to $\mathrm{bi}(\{p_i,p_j,p_k\},\lambda)=B(p_i,\lambda)\cap B(p_j,\lambda)\cap B(p_k,\lambda)$. We have that
$$\|x-p\|=\|t_i(x-p_i)+t_j(x-p_j)+t_k(x-p_k)\|\leq \lambda$$
and $x\in B(p,\lambda).$ As a consequence,  $\mathrm{bi}(\{p_i,p_j,p_k\},\lambda)=\mathrm{bi}(\{p_i,p_j,p_k,p\},\lambda)$.

Let us assume that $p\in \mathrm{conv}(\widehat{p_ip_j},\overline{p_ip_j})$ and consider $S(p_i,\lambda)\cap S(p_j, \lambda)$. By Lemma \ref{twocircles},  this intersection consists of two connected components $A_1$, $A_2$.

If $\lambda\geq \mathrm{diam}(K)$, then the intersection of the circles  consists of  not only one  component.

If the normed plane is strictly convex and one of the two components, for instance $A_2$, is the empty set, then $A_1$ is the point $\{\frac{p_i+p_j}{2}\}=B(p_i,\lambda)\cap B(p_j, \lambda)$, and $\frac{p_j+p_j}{2}$ is the center of the minimal arc meeting $p_i$ and $p_j$. Obviously, $\|p-\frac{p_i+p_j}{2}\| \leq \lambda$ and $B(p_i,\lambda)\cap B(p_j, \lambda)\subset B(p,\lambda)$.




Let us assume that $A_1$ and $A_2$ are two different and non-empty sets. Let $x$  be the center of the arc $\widehat{p_ip_j}$, and $x'$ be a point belonging to $S(p_i,\lambda)\cap S(p_j, \lambda)$ such that $x\in A_1$ and $x'\in A_2$.  The four points $p_i, x, p_j, x'$ are the vertices of a quadrangle; the line $\langle p_i, p_j \rangle$ separates $x$ and $x'$; and $\|x-p\| \leq \lambda$. By Lemma \ref{twocircles} (with the discs cited in this lemma and centered at $x$ and $x'$) we conclude that $\|x'-p\|\leq \lambda$. Therefore, both points $x$ and $x'$ belong to $B(p,\lambda)$. Two cases are possible:

Case 1: $\mathbb{M}^2$ is strictly convex. By $(4)$ in Lemma \ref{3.2}, the boundary of $B(p_i,\lambda)\cap B(p_j, \lambda)$ consists of
minimal arcs meeting $x$ and $x'$,  and $(3)$ in the same Lemma \ref{3.2} implies that these minimal arcs belong to $B(p,\lambda).$ Therefore,
$B(p_i,\lambda)\cap B(p_j, \lambda)\subset B(p,\lambda).$

Case 2: $\mathbb{M}^2$ is not strictly convex, but $\lambda\geq \mathrm{diam}(K)$.

Subcase 2.1. $p\in \widehat{p_ip_j}$. We use the following statement: let $H$ be a line, and $p$ be a point with $p\not\in H$. Let $q$ be a point on $H$ such that $\langle p, q\rangle$ is Birkhoff orthogonal to $H$. If $q_1$ and $q_2$ lie on the same half-line of $H$  with respect to $q$, and $q_1$ is between $q$ and $q_2$, then $\|p-q_1\|\leq \|p-q_2\|$.

By Lemma \ref{twocircles}, if $H$ is a line through $x$ and parallel to $\langle p_i, p_j\rangle$, then $B(p_i,\lambda) \cap B(p_j,\lambda)$ belongs to the half-plane bounded by $H$ which contains the points $p_i$ and $p_j$.

 Let $H_i$ be  a line through $p$ which is Birkhoff orthogonal to $\langle p_j, p\rangle$ (the case that $\langle p_j,p\rangle$ is orthogonal to $\langle p_j,p\rangle$ should be considered in an extra way). If $H_i^+$ is the half-plane bounded by $H_i$ and containing $p_i$, then $x$ belongs to the half-plane  opposite to $H_i^+$. Analogously, $H_j$ is a line Birkhoff orthogonal to $\langle p_i,p\rangle$, and $H_j^+$ is the half-plane bounded by $H_j$ and containing $p_j$. Let $y$ be a point from $ B(p_i, \lambda)\cap  B(p_j, \lambda)$. If $y\in H_i^+$, then
$$\|y-p\|\leq \|y-p_j\|\leq\lambda,$$
and therefore $y\in B(p, \lambda)$. If $y\in H_j^+$, then
$$\|y-p\|\leq \|y-p_i\|\leq\lambda.$$
Thus we have only to prove: when $y\in \big(\mathbb{R}^2\setminus (H_i^+\cup H_j^+)\big )\cap\big (B(p_i, \lambda)\cap  B(p_j, \lambda)\big)$, then $y\in B(p, \lambda)$. In order to prove this, we use Lemma \ref{curvebisector}. Let $\mathfrak{B}^\ast(p_i, p)$ be the curve from Lemma \ref{curvebisector} passing through $x$ ($x\in \mathfrak{B}(p_i, p)$). Analogously,  $\mathfrak{B}^\ast(p_j, p)$ passes through $x$. Let  $H$ be a line through $x$ and parallel to $\langle p_i, p_j\rangle$, and $H^+$ be the half-plane bounded by $H$ which contains the points $p_i, p_j, p$. Let $\mathfrak{B}^\ast(p_i, p)^+$ be the open part of the plane bounded by $\mathfrak{B}^\ast(p_i, p)$ and belonging to $H^+$ such that for every $u\in \mathfrak{B}^\ast(p_i, p)^+$ the inequality $\|u-p_i\|\leq \|u-p\|$ holds. Analogously, $\mathfrak{B}^\ast(p_j, p)^+$ is the open part of the plane bounded by $\mathfrak{B}^\ast(p_j, p)$ and  belonging to $H^+$  such that for every $u\in \mathfrak{B}^\ast(p_j, p)^+$ the inequality  $\|u-p_j\|\leq \|u-p\|$ holds.

Assume that there exists $y\in \big(\mathbb{R}^2\setminus (H_i^+\cup H_j^+)\big )\cap\big (B(p_i, \lambda)\cap  B(p_j, \lambda)\big)$ such that $\|y-p\|>\lambda$. Then $y\in \mathfrak{B}^\ast(p_i, p)^+\cap \mathfrak{B}^\ast(p_j, p)^+$. But since $\mathfrak{B}^\ast(p_i, p)$ belongs to the cone with apex $x$  and passing through $p_i$ and  $p$, and $\mathfrak{B}^\ast(p_j, p)$ belongs to the cone with apex $x$ and passing through $p_j$ and  $p$, the open regions $\mathfrak{B}^\ast(p_i, p)^+$ and $\mathfrak{B}^\ast(p_j, p)^+$ do not have points in common.

In conclusion, if $p\in \widehat{p_ip_j}$, then $B(p_i,\lambda) \cap B(p_j,\lambda) \subset B(p,\lambda)$

Subcase 2.2. Let us assume that  $p\in \mathrm{conv}(\overline{p_i,p_j}, \widehat{p_i,p_j})$. Then the ray emanating from $p_i$ and passing through $p$ intersects $\widehat{p_i,p_j}$ in a point $p'$. Let $p=\mu p_i+(1-\mu) p'$, where $\mu\in [0, 1]$. By Subcase 2.1, we have that  $B(p_i, \lambda) \cap B(p_j, \lambda) \subset B(p', \lambda)$. Let $x\in B(p_i, \lambda) \cap B(p_j, \lambda)$. Then
$$\|p-x\|=\|\mu p_i+(1-\mu) p'-x\|=\|\mu p_i-\mu x+(1-\mu) p'-(1-\mu)x\|\leq$$ $$ \mu \|p_i-x\|+(1-\mu)\|p'-x\|\leq \mu\lambda+(1-\mu)\lambda=\lambda,$$
and therefore
$B(p_i,\lambda) \cap B(p_j,\lambda) \subset B(p,\lambda)$.

In conclusion, if $p$ is not a vertex of  $\mathrm{bh}(K,\lambda)$, then $B(p,\lambda)$ has no influence on the construction of $\mathrm{bi}(K,\lambda)$, and every arc of $\mathrm{bi}(K,\lambda)$ is generated by a vertex of $\mathrm{bh}(K,\lambda)$, which is the first part of the theorem.


Let $\{p_1,p_2,\dots,p_k\}$ be the set of vertices of  $\mathrm{bh}(K,\lambda)$ ordered as they are obtained by the process described by the proofs of Proposition \ref{proposition3} and Proposition \ref{proposition5.5} in \cite{Ma-Ma-Sp}, and $\{x_1,x_2,\dots,x_{k},x_{k+1}\}$ be  the set of points managed in the same process, with the same order, but adding the endpoint $x_{k+1}:=x_1$.

We know that $\|x_i-p_j\|\leq \lambda$ for $i\in\{1,\dots,k\}$ and $j\in\{1,\dots,n\}$, and  therefore $x_i\in \cap_{i=1}^n B(p_j,\lambda)$. Observing the process, we realize that the arc $\widehat{x_i x_{i+1}}$ (in the way described by $z$ from $x_i$ to $x_{i+1}$ along $S(p_i,\lambda)$) is  separated from its center $p_i$ by the line $\langle p_i, p_{i+1}\rangle$, and therefore it is a minimal circular arc of radius $\lambda$. By  part (3) of Lemma \ref{3.2}, $\widehat{x_ix_{i+1}}$ belongs to every $B(p_j,\lambda)$,  justifying that
$$\mathrm{bi}(K,\lambda)=\mathrm{conv}\left( \bigcup_{i=1}^k \widehat{x_{i}x_{i+1}}\right),$$
where every arc $\widehat{x_ix_{i+1}}$ is generated by a vertex $p_i$ of $\mathrm{bh}(K,\lambda)$, and every vertex $x_i$ of $\mathrm{bi}(K,\lambda)$ is the center of an arc $\widehat{p_ip_{i+1}}$ belonging to the boundary of $\mathrm{bh}(K,\lambda$).
\end{proof}

\bigskip

\section{The algorithm for the 2-center problem}

Given $r_1>r_2>0$, the planar 2-center decision  problem asks whether a set $K$ of $n$ points  in a normed plane can be covered by two discs of radius $r_1$ and $r_2$, respectively. Without loss of generality, we can assume that $r\equiv r_1>r_2=1.$ Of course, the consideration makes  only  sense when the diameter of $K$ is larger than the radius $r_1$, since otherwise a disc of radius $r_1$ covers $K$. Hershberger and Suri \cite{H-S} published an algorithm solving the problem in $ O(n^2\ \mathrm{log}\, n)$ time  and $O(n)$ space in the Euclidean plane. Hershberger \cite{H2} presented an algorithm that runs in $O(n^2)$ time and space. In this section we revise this last algorithm for adapting it to every strictly convex normed plane, proving some statements which are trivial in the Euclidean plane but not in general.

Hershberger constructs the full arrangement of discs of radius $r$ (shortly, called \emph{$r$-discs}) centered at points of $K$. For each $r$-circle $C$ of the arrangement, it is explored whether the points not covered by the $r$-disc can be covered by a separate unit disc. This is carried out moving a reference point $p$ along the boundary of every $r$-circle in four disjoint $90^o$ sweeps, keeping track of the set $F$ of points that are farther than a distance $r$ away. For each sweep, the set $F$ is separated in two disjoint sets of points, $D$ and $A$: the set $D$ contains all points farther than $r$ from the initial position of $p$, and $A$ is empty. As $p$ moves, delete from $D$ any point that becomes less than a distance $r$ away; add to $A$ any point that becomes more than a distance $r$ away. We have $F=A\cup D$ during the $90^o$ sweep.  After each change to $A$ and $D$, test whether $\mathrm{bi}(A,1)$ and $\mathrm{bi}(D,1)$ intersect. The point $p$ and every point of this (eventually) non-empty intersection become the centers of a solution pair of discs for the 2-center problem.

Step 1 and Step 2 below describe the global structure of the algorithm  for the 2-center problem in the Euclidean plane  presented in \cite{H2}. After presenting every step, we prove the necessary statements for extending the algorithm to a strictly convex normed plane.\vspace{0.3cm}

\textbf{Step 1}. \textsl{Building  the arrangement of discs with radius $r$ centered at the points of $K$}.


Supporting the above statement and completing Step 1 in a strictly convex normed plane one only needs the following lemma.

\begin{lem}\label{lem8}
Let $\mathbb{M}^2$ be a strictly convex normed plane. If $K$ is a set of $n$ points in $\mathbb{M}^2$, then, building  the arrangement of circles with radius $r>0$ centered at  the points of $K$ takes  $O(n^2\log\, n)$ time and $O(n^2)$ space.
\end{lem}

\begin{proof}
The circles are well behaved, i.e., the total number of intersections is $O(n^2)$ and the calculation of the intersection points of two circles is a basic operation in our process and takes constant time. Thus, the arrangement can be constructed in $O((n+k)\log\, n)$ time using (the output sensitive) plane-sweep algorithm (see \cite{fhw-caa-12}), where $k = O(n^2)$ is the number of intersections between the circles. The space is obviously $O(n^2)$.
\end{proof}

\begin{rem} Chazelle and Lee \cite{Ch-L} proved that for given $n$ points $p_1,...,p_n$  and a fixed disc of radius $r$, the arrangement of Euclidean circles with radius $r$ centered at these points can be constructed in $O(n^2)$ time and space. The algorithm of Chazelle and Lee could be adapted for strictly convex norms if certain geometric and graph-theoretic properties of arrangements of Euclidean circles from \cite{Ch-L} will be proved for norm circles.
\end{rem}

\textbf{Step 2.} \textsl{Moving a reference point $p$ along the boundary, keeping track of the set $F$ of points that are farther than a distance $r$ away}.

For describing precisely the idea from \cite{H2}, let $p_{\theta}$ be a parametrization of  $C$ ($\theta\in [0,360^o]$) and  consider four disjoints sweeps in the circle. Let $[\theta_1, \theta_1+90^o)$ be one of those sweeps. For every $\theta\in[\theta_1, \theta_1+90^o)$, define the following variable sets:

    $    F_\theta$: the points of $K$ which do not belong to the disc centered at $p_{\theta}$;

    $D_\theta$: the points of $K$ which do not belong to the disc centered at $p_{\theta}$ and do not belong to the discs centered at other  previous (in the oriented sense  of the parametrization of the circle) points.

    $A_\theta$: the points of $K$ which do not belong to the disc centered at $p_{\theta}$, but each of them belongs to some disc centered at other  previous (in the oriented sense of the parametrization of the circle) points.

Obviously, $F_{\theta}=D_{\theta}\cup A_{\theta}$ for every $\theta$. Move $p_{\theta}$ along the sweep starting with $\theta=\theta_1$. The sets $F_{\theta}$, $D_{\theta}$ and $A_{\theta}$ may change. At the beginning we have $F_{\theta}=D_{\theta}$, and $A_{\theta}$ is empty. When $p_{\theta}$ moves, some points  are possibly  deleted in $D_{\theta}$, because they leave $F_{\theta}$. Nevertheless, some points possibly begin to belong to $A_{\theta}$ when they begin to belong (again or for the first time) to $F_{\theta}$.

For each circle $C$, for each of the four $90^o$ arcs that cover $C$:

\textbf{Step 2(a).} Find the order of insertions and deletions to $A_{\theta}$ and $D_{\theta}$ in $O(n)$ time by walking along the boundary of $C$.

\textbf{Step 2(b).} Process the insertions to $A_{\theta}$ in sequence, maintaining $\mathrm{bi}(A_{\theta},1)$. Record the changes to $\mathrm{bi}(A_{\theta},1)$ in a transcript. As Theorem \ref{theo2inHS} shows, in  every strictly convex normed  plane it is possible to maintain $\mathrm{bi}(A_{\theta},1)$ as $A_{\theta}$ grows in $O(|  A  |)$ total time, where $|A|$ is the final cardinality of $A_{\theta}$.

\textbf{Step 2(c).} Partition of  the initial set $D_{\theta_1}$ into a static set $Z$ of points that will not be deleted during the sweep, and a dynamic set $Y_{\theta}$ of points that will be deleted. Use the algorithm of Theorem \ref{theo2inHS} to compute a change-transcript for $\mathrm{bi}(Y_{\theta},1)$, working in time-reversed order. Using Theorem \ref{theo3inHS}, combine this with $\mathrm{Z}$ to get a change-transcript for $\mathrm{bi}(D_{\theta},1)$.

\textbf{Step 2(d).} Play the transcripts for $A_{\theta}$ and $D_{\theta}$ simultaneously, both in forward time order (the reverse of the construction order for $D_{\theta}$). Test whether $\mathrm{bi}(A_{\theta},1)$ and $\mathrm{bi}(D_{\theta},1)$ overlap at any point during the playback. Theorem \ref{theo4inHS} proves that this can be done in $O(n)$ time altogether.

\bigskip

The following results allow us to present our Theorem \ref{theo5inHS} at the end of this section. They extend  Theorems 2, 3 and 4 in \cite{H2} for the Euclidean plane to  strictly convex normed  planes.

\begin{lem}\label{lem2H2} 
Let $\mathbb{M}^2$ be a strictly convex normed plane, and $p_1$, $p_2,\dots,p_m$ be the points added to $A_{\theta}$ in order of addition when $p$ moves during its $90^o$ sweep. Suppose that each arc on the boundary of $\mathrm{bi}(\{p_1,p_2,\dots,p_j\},1)$ is labelled with an integer $i$ if the arc belongs to the disc centered at $p_i$. Then, for any $j$, the labels of the arcs of $\mathrm{bi}(\{p_1,p_2,\dots,p_j\},1)$ form a circular sequence with no duplication, one local maximum, and one local minimum.
\end{lem}

\begin{proof}
Lemma 1 in \cite{H2} proves the statement in the Euclidean case making use of the following arguments:

1) \textsl{Every arc from the boundary of $\mathrm{bi}(K,r)$ is generated by a vertex on the boundary of $\mathrm{bh}(K,r)$, and every vertex on $\mathrm{bi}(K,r)$ corresponds to an arc from the boundary of $\mathrm{bh}(K,r)$.}

The above statement is proved for every strictly convex normed plane in our Theorem \ref{bi-bh} in Section \ref{section3}.

2) \textsl{If $r\geq 1$, then $\mathrm{bh(K,r})\subseteq \mathrm{bh}(K,1).$}

This is proved for any normed plane (see Proposition 3.1 in \cite{Ma-Ma-Sp}).

3) \textsl{Let $B(x_1,1)$ and $B(x_2,1)$ be two discs in a normed plane
$\mathbb{M}^2$ whose intersection has non-empty interior. Let $t_3$ be a
point  belonging to $S(x_1,1)\cap S(x_2,1)$. Let us consider a third circle $S$ of
radius $1$ containing $t_3$. This third circle cannot  simultaneously contain points of  both the arcs which form the boundary of $B(x_1,1)\cap B(x_2,1)$.}


\begin{figure}[ht]
\begin{center}
\includegraphics[width=10cm]{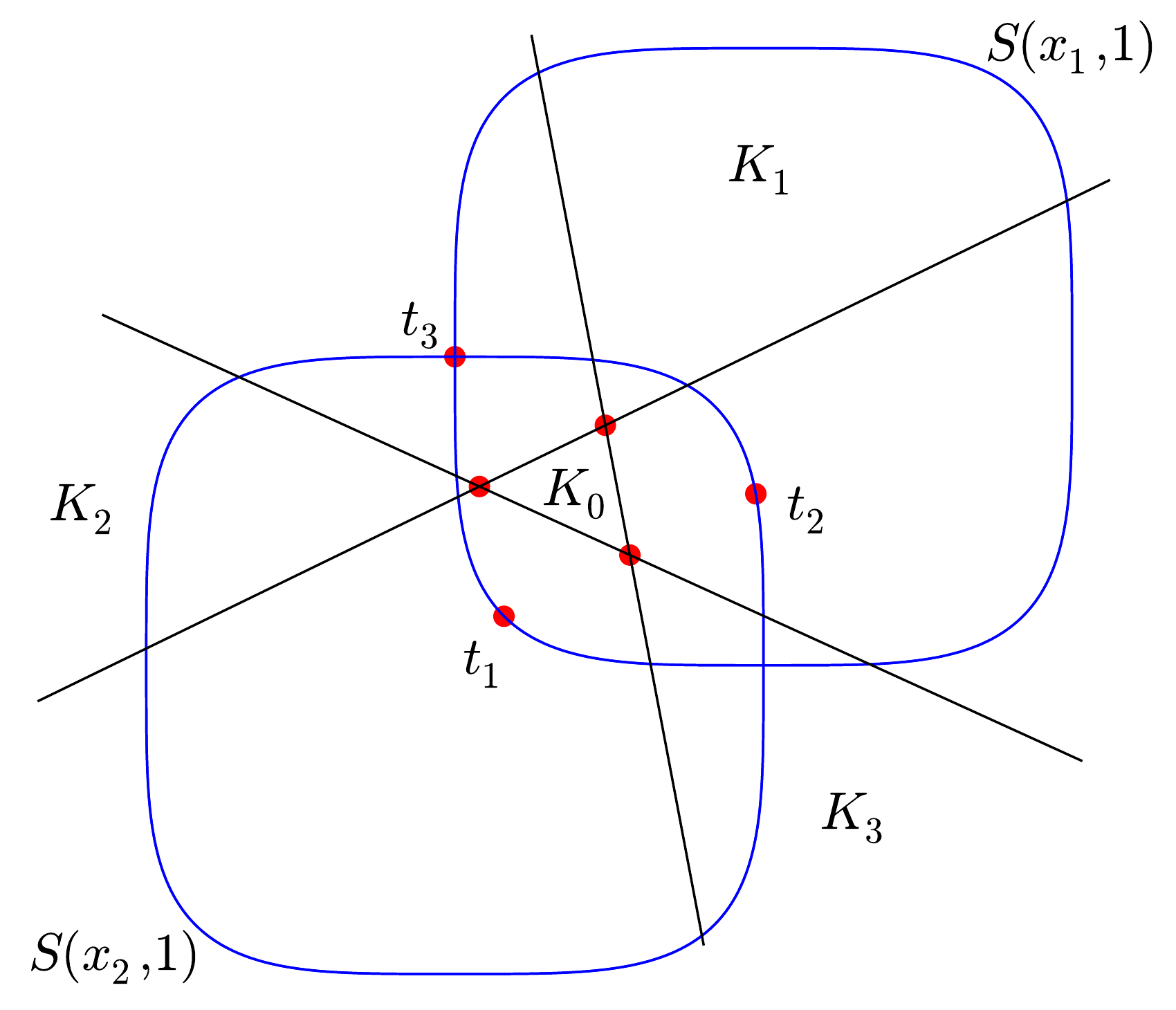}\\
\caption{Regions of centers of circles passing through $t_1, t_2, t_3$}\label{intersection three balls}
\end{center}
\end{figure}

Let us suppose that $t_1$ and $t_2$ are points belonging to $S$ and to the boundary of $B(x_1,1)\cap B(x_2,1)$, such that $t_1\in S(x_1,1)$ and $t_2\in S(x_2,1)$. Let us consider the triangle with vertices $t_i$, $i=1,2,3$. Using the
same notation as in \cite{Al-Ma-Sp}, the center of the third circle can
belong to the regions $K_i$, $i=0,1,2,3$ (see Figure \ref{intersection three balls}), bounded by the
lines meeting the middle points of the sides of this triangle. By Lemma \ref{3.0} in Section \ref{section3}, the centers of two circles in a strictly convex normed plane with non-empty intersection are situated in different half-planes bounded by the line meeting the common points. If the center of $S$ belongs to $K_3$, by Gr\"{u}nbaum \cite{Grue1}
 (see Lemma \ref{twocircles}) the arc $\widehat{t_2t_3}$ (counterclockwise) in the circle $S$ is contained in the convex hull defined by $\overline{t_2t_3}$ and the arc $\widehat{t_2t_3}$ (counterclockwise) of the circle $S(x_2,1)$. In addition, the arc $\widehat{t_3t_2}$ (counterclockwise) of $S(x_2,1)$ is contained in the  convex hull of $\overline{t_3t_2}$ and the arc $\widehat{t_3t_2}$ (counterclockwise) of $S$. Since $t_1$ belongs to the convex hull defined by $\overline{t_3t_2}$ and the arc $\widehat{t_3t_2}$ (counterclockwise) of $S(x_2,1)$, we conclude that $t_1$ cannot belong to $S$. The arguments are similar when the center of $S$ is from $K_i$ for $i=0,1,2$.

4) \textsl{If a point $b$ belongs to the convex hull of three points $\{a,c,d\}$, then $b$ is not a vertex of the ball hull of $\{a,b,c,d\}$ of radius $r$}.

  This is obviously true in arbitrary  normed planes.

Therefore, the proof presented  in \cite{H2} is also valid for every strictly convex normed  plane by  using the arguments above.
\end{proof}

\begin{theo} \label{theo2inHS}
We can maintain $\mathrm{bi}(A_{\theta},1)$ as $A_{\theta}$ grows in $O(|  A  |)$ total time, where $|A|$ is the final cardinality of $A_{\theta}$.
\end{theo}

\begin{proof}
Using the statement of Lemma \ref{lem2H2} for the Euclidean case, this result  is proved  in \cite{H2} (noted there as Theorem 2). With our Lemma \ref{lem2H2}, the same proof is extended to strictly convex normed planes.
\end{proof}

\begin{theo} \label{theo3inHS}
Let $\mathbb{M}^2$ be a strictly convex normed plane. Let $Z$ be a fixed set of points, and  $Y_{\theta}$ be a dynamic set subject to deletions only. Suppose we know the $x$-sorted order of $Y_{\theta}$ and $Z$, and suppose we are given a change-transcript for $\mathrm{bi}(Y_{\theta},1)$. Then we can compute a change-transcript for $\mathrm{bi}(Y_{\theta}\cup Z,1)$ with $O(|Y_{\theta}|+|Z|)$ additional work.
\end{theo}

\begin{proof}
The proof of Theorem 3 in \cite{H2} is also valid for strictly convex normed planes.

\end{proof}

Theorems \ref{theo2inHS} and \ref{theo3inHS} allow us to compute $D_{\theta}$ in $O(|D_{\theta}|)$ time.


\begin{theo}\label{theo4inHS}
Let $\mathbb{M}^2$ be a strictly convex normed  plane. Let $A_{\theta}$ and $D_{\theta}$ be two sets of points subject to insertions  and deletions only, respectively. Suppose that some oracle maintains the boundaries of $\mathrm{bi}(A_{\theta},1)$ and $\mathrm{bi}(D_{\theta},1)$ in a linked-list form, and that at each  update it  gives us a pointer to the changed edges. Then we can detect when, if ever, $\mathrm{bi}(A_{\theta},1)$ and $\mathrm{bi}(D_{\theta},1)$ intersect, using $O(|A_{\theta}|+|D_{\theta}|)$ time altogether, where $|A_{\theta}|$ and $|D_{\theta}|$ are the maximum cardinalities of $A_{\theta}$ and $D_{\theta}$.
\end{theo}

\begin{proof}
We use the proof of Theorem 4 from \cite{H2} for the Euclidean case. When  $\mathrm{bi}(A_{\theta},1)$ and $\mathrm{bi}(D_{\theta},1)$ do not overlap, then there exists an inner common tangent between them.  Assume that the inner common tangent is clockwise-heading: as it heads from $\mathrm{bi}(A_{\theta},1)$ towards $\mathrm{bi}(D_{\theta},1)$, it leaves $\mathrm{bi}(A_{\theta},1)$ on its right side. These arguments are used in Hershberger's proof:
\begin{enumerate}
\item If  $\mathrm{bi}(A_{\theta},1)$ or $\mathrm{bi}(D_{\theta},1)$ changes, the inner common tangent is unchanged, unless one of the edges to which the tangent is incident is affected by the change.
\item If a point $p$ is added to $A_{\theta}$, the region $\mathrm{bi}(A_{\theta},1)$ shrinks. If the inner common tangent moves, it will be tangent to $\mathrm{bi}(A_{\theta},1)$  at the disc contributed by $p$; on $\mathrm{bi}(D_{\theta},1)$, the point of tangency will move counterclockwise. Find the new tangent point on $\mathrm{bi}(D_{\theta},1)$ by walking along the boundary of $\mathrm{bi}(D_{\theta},1)$ from the old tangent point.
\item If a point $q$ is deleted from $D_{\theta}$, the region $\mathrm{bi}(D_{\theta},1)$ grows. If the inner common tangent moves, the point of tangency on
$\mathrm{bi}(A_{\theta},1)$ moves clockwise; find it by walking along $\mathrm{bi}(A_{\theta},1)$ from the old tangent point. On $\mathrm{bi}(D_{\theta},1)$, the new tangent point lies on a chain of edges revealed by the deletion of $q$'s disc. We could find this new tangent point by walking from either end of the chain; we choose to walk counterclockwise from the clockwise end.

\item Searches for the point of tangency on $\mathrm{bi}(A_{\theta},1)$ move clockwise through the edges; searches on $\mathrm{bi}(D_{\theta},1)$ move counterclockwise. Because the searches never back up from one edge to a previous edge, the total search time is $O(|A|+|D|)$.
\end{enumerate}
Due to the structure of the boundary of the ball intersection, and its
non-increasing property when points are added, all these arguments above
are also valid for any strictly convex normed plane. Therefore, the theorem
is true for such planes.
\end{proof}

Using Lemma \ref{lem8} and Theorems \ref{theo2inHS}, \ref{theo3inHS}, and \ref{theo4inHS}, we finally obtain  our main result in this section.

\begin{theo}\label{theo5inHS}
In any strictly convex normed plane, the generalized planar 2-center decision problem can be solved in $O(n^2\log\, n)$  time and $O(n^2)$  space.
\end{theo}

\begin{rem} We can rewrite our Theorem \ref{theo5inHS} using the following from \cite{F-H-K-T-W-W}. Let there be given a set of $n$ curves, where each pair of curves can have
at most $s$ intersection points, for a constant $s$. In this case, the zone complexity
of a single curve is $O(\lambda_{s+2}(n))$, where $\lambda_{\sigma}(k)$ denotes the maximal length of a
Davenport-Schinzel sequence of $k$ elements with order $\sigma$. Thus, the total construction time is $O(n\lambda_{s+2}(n))$, while the complexity of the arrangement is of course $O(n^2)$. Therefore, our Theorem \ref{theo5inHS} can be formulated as follows: the  generalized planar 2-center decision problem can be solved
in $O(\max\{ n\lambda_{4}(n) , n^2\})$ time and $O(n^2)$ space in any strictly convex normed plane. It is important to note that $\lambda_{\sigma}(k)$ is almost linear in $k$ for small values of $\sigma$ (for  Davenport-Schinzel sequences see \cite{Ag-Sh-Sh} and \cite{N2}).
\end{rem}

\section*{Acknowledgments}

For helpful hints the authors wish to thank Efi Fogel.

\vspace{1cm}
\begin{tabular}{l}
Pedro Mart\'{\i}n\\
Departamento de Matem\'{a}ticas,\\
Universidad de Extremadura,\\
06006 Badajoz, Spain\vspace{0.1cm}\\
E-mail: pjimenez@unex.es
\end{tabular}\vspace{0.3cm}

\begin{tabular}{l}
Horst Martini\\
Fakult\"at f\"ur Mathematik, TU Chemnitz\\
D-09107 Chemnitz, GERMANY\vspace{0.1cm}\\
E-mail: horst.martini@mathematik.tu-chemnitz.de
\end{tabular}\vspace{0.3cm}

\begin{tabular}{l}
Margarita Spirova\\
Fakult\"at f\"ur Mathematik, TU Chemnitz\\
D-09107 Chemnitz, GERMANY\vspace{0.1cm}\\
E-mail: margarita.spirova@mathematik.tu-chemnitz.de
\end{tabular}

\end{document}